\documentclass[12pt]{article}
\usepackage{graphicx}
\usepackage{amsmath,amsthm}
\usepackage{amssymb}
\usepackage[latin1]{inputenc}
\usepackage{enumerate}
\usepackage{slashed}
\usepackage{url}
\usepackage{pb-diagram}
\usepackage{slashed}
\usepackage{color}
\theoremstyle{plain}
\newtheorem{theo}{Theorem}

\newtheorem{obs}{Observation}
\newtheorem{define}[theo]{Definition}

\newcommand{\End}{\mathrm{End}}
\newcommand{\rmx}{\mathrm{x}}

\newcommand{\ord}{\mathrm{ord}_G}
\newcommand{\oord}{\mathrm{ord}_{pG}}
\newcommand{\cord}{\mathrm{ord}_{cG}}
\newcommand{\rord}{\mathrm{ord}_{rG}}

\newcommand{\tr}{\mathrm{tr}}

\newcommand{\diver}{\mathrm{div}}
\newcommand{\vol}{\mathrm{vol}}
\newcommand{\Tr}{\mathrm{Tr}}
\newcommand{\cE}{\mathcal{E}}

\newcommand{\cD}{\mathcal{D}}

\newcommand{\grad}{\mathrm{grad}}

\newcommand{\cA}{\mathcal{A}}
\newcommand{\cR}{\mathcal{R}}

\newcommand{\cO}{\mathcal{O}}
\newcommand{\bbC}{\mathbb{C}}
\newcommand{\cL}{\mathcal{L}}

\newcommand{\ddet}{\mathrm{Det}}
\newcommand{\eps}{\epsilon}

\newcommand{\bbR}{\mathbb{R}}

\newcommand{\id}{\mathrm{Id}}

\newcommand{\p}{\partial}

\title{\textbf{Some Modifications of Getzler's Grading Technique}}
\author{Andr\'es Larra\'in-Hubach\footnote{University of Dayton,   OH 45469. alarrainhubach1@udayton.edu}}
\begin{document}
\maketitle

\begin{abstract}
This paper reviews the grading technique developed by Getzler to prove the local index theorem and shows how to adapt it to compute the leading terms of asymptotic expansions of traces of heat kernels in two other situations.
\end{abstract}
\tableofcontents
\section{Introduction}
The main objective of this work is to explain how the grading technique developed by Ezra Getzler \cite{get1} to prove the local index theorem can be used to study some problems of the following type: given  a family of geometric Dirac operators $\cD_s$ that depends on a certain real parameter $s$, what can be said about the leading order terms of the asymptotic expansion in powers of $s$ of the heat kernel $e^{-\frac{t}{s}\cD_s}(x,x)$ evaluated on the diagonal?

Understanding heat kernels, their traces and their values on the diagonal is usually an important technical step that must be cleared before proving results like the holomorphic Morse inequalities \cite[Sec. 1.7]{mm}, the asymptotics of the spectral counting function \cite[Prop. 4.2]{sav} and, more classically, the local index theorem \cite[Ch. 4]{bgv}. Here we show how  minor adaptations of the grading technique can be used to  compute the leading terms in all of these expansions. The elementary proofs presented here are similar to each other and display the grading technique as  a ``unifying theme".

In section 2, we review classical results about geometric Dirac operators, the asymptotic expansions of their heat kernels and the original definition of Getzler's grading technique. Finally, we review the well-known proof of the local index theorem for a twisted Dirac operator over a spin manifold. Minor modification of the ideas in this section will be used afterwards. There are no new results in this section.

In section 3, we consider the following problem: given a complex manifold $M$ and a Hermitian bundle $\cE=\Lambda^*T^{(0,1)*}M\otimes L^{\otimes P}\otimes E$ with Clifford connection induced from the Chern connections on $L$ and $E$, compute the leading order of the asymptotic expansion of the diagonal values of $e^{-\frac{u}{p}(\cD^\cE)^2}(x,x)$ as $p\to\infty$, where $\cD^\cE$ is the associated geometric Dirac operator. The answer, stated precisely  in section \ref{exp1}, is \cite[Thm. 1.6.1]{mm} and the argument presented there  uses a method called \textit{Analytic Localization} which was originally developed by Bismut-Lebeau \cite{bis}. This technique is powerful and can be used in several different contexts but it is rather challenging to understand and relies on a considerable amount of machinery from functional analysis. In the opinion of the author, the new proof presented here of the same result is considerably easier since it only uses the classical expansion of the heat kernel and a small modification of the grading technique.

In section 4, we present a proof of \cite[Thm. 4.1]{sav}, explained in section \ref{exp2}. Again, our argument is different from the one presented in that paper, in that  we modify the grading scheme and use the classical asymptotic expansion of the heat kernel.

It is the expectation of the author that the reader will appreciate the power of the grading technique as a useful and adaptable method to solve several problems concerning heat kernels. For future work, it would be interesting to use the grading technique to compute the leading terms of asymptotic expansions of terms like $\Tr (\cD_se^{-t\cD_s^2})$ as $s$ goes to infinity. 

\section{Preliminaries}
In this section we fix notation and state the main properties and theorems to be used in the following sections. We use Einstein's summation whenever indices are repeated. All manifolds are assumed compact, orientable and without boundary. A good reference for this section is \cite[Sec. 2.3]{ros}.
\subsection{Normal Coordinates and Synchronous frames}\label{basic}
Let $(M,g)$ be a Riemannian manifold. The Levi-Civita connection on $TM$ is denoted  by $\nabla$. Using normal coordinates $\{\rmx_l\}_{l=1}^n$ around $y$, we can identify $y$ with $0\in\bbR^n.$ This will be done implicitly in what follows. The frame $\{\frac{\p}{\p \rmx_l}\}_{l=1}^n$ is orthonormal at $y$. In fact, in these coordinates $g_{ij}(\rmx)=g(\frac{\p}{\p \rmx_i},\frac{\p}{\p \rmx_j})(\rmx)=\delta_{ij}+\cO(\rmx^2)$ and the symbols of the Levi-Civita connection symbols satisfy $\Gamma_{ij}^k=\cO(\rmx).$

The local radial function is defined by $r^2(\rmx)=\sum_{j=1}^n \rmx_j^2.$ It is a consequence of Gauss lemma  that $r(\rmx)=d(x,y)$ where $d$ is the distance function induced by the metric $g$. The radial vector field is defined by $r\p_r=\rmx_l\frac{\p}{\p \rmx_l}.$ Taking the  $\{\frac{\p}{\p \rmx_l}\}_{l=1}^n$ at $y$ and  using parallel translation along geodesics that emanate from $y$ results  in a local orthonormal frame that we denote by $\{e_l\}_{l=1}^n$ with dual orthonormal coframe denoted by $\{\omega^l\}_{l=1}^n.$

Let $E\to M$ be a Hermitian vector bundle with unitary connection $\nabla^E=d+A$. Take an orthonormal frame of the fiber of $E$ over $y\in M$ and extend it to a local frame  using again parallel translations along geodesics emanating from $y$. In this local frame, called synchronous, the connection symbols satisfy
\begin{equation}\label{sync}
\nabla^E_i=\frac{\p}{\p \rmx_i}-\frac{1}{2}\rmx^jF(\frac{\p}{\p \rmx_i},\frac{\p}{\p \rmx_j})+\cO(\rmx^2)=\frac{\p}{\p \rmx_i}-\frac{1}{2}\rmx_jF_{ij}+\cO(\rmx^2),
\end{equation}
where $F$ denotes the curvature two-form of $\nabla^E.$ (See \cite[Prop. 12.22]{roe}).
\subsection{Geometric Dirac Operators}
In this section, $(M,g)$ is a closed Riemannian  manifold of $\dim_{\mathbb{R}}M=n=2m.$   We  denote by $CL(TM)\otimes \mathbb{C}$ the complexified Clifford algebra bundle of $TM.$ Our conventions for Clifford algebras are the same as those of \cite{roe}. Additionally, we have a Hermitian bundle $\cE=\cE^+\oplus \cE^-\to M$ with metric $h^\cE$ and compatible connection $\nabla^\cE$ satisfying the following properties:
\begin{itemize}
\item The bundle $\cE$ is a Clifford bundle. In other words, there is a homomorphism of complex algebras $c: CL(TM)\otimes \mathbb{C}\to \End\,\cE$ called Clifford multiplication.  The Clifford multiplication by a tangent vector field $X$ is denoted by $c(X)\in \Gamma(\End \,\cE)$. One also assumes that $c(X)$ is skew-symmetric with respect to $h^\cE$ and $c(X):\cE^\pm\to\cE^\mp.$
\item The connection $\nabla^\cE$ preserves the splitting $\cE=\cE^+\oplus \cE^-.$
\item The connection $\nabla^\cE$ is a Clifford connection. In other words, $\nabla^\cE_X(c(Y)\psi)=c(\nabla_XY)\psi+c(Y)\nabla^\cE_X\psi$, where $X,Y$ are vector fields on $M$ and $\psi\in\Gamma(\cE).$
\end{itemize}
On $y\in M$, we have that $\cE_y\cong S\otimes W_y$, where $S$ is the spinor space and $W$ is the so-called twisting bundle. The $S$ glue to a globally defined vector bundle when $M$ is spin. Furthermore, the following isomorphism of complex algebras holds \cite[Ch. 4]{roe}
\begin{equation}\label{spinor}
\End(\cE_y)\cong (CL(T_yM)\otimes\bbC)\otimes \End(W_y).
\end{equation}

The algebra $CL(T_yM)\otimes\bbC$ is generated by $c(e_j)=c^j$ for $j=1\ldots, n.$ The $\{e_j\}$ is the local orthonormal frame defined in section \ref{basic}. Using normal coordinates and a synchronous frame on $\cE$ we get the following local expression for the Clifford connection $\nabla^\cE$:
\begin{equation}\label{nablaloc}
\nabla^\cE_{\frac{\p}{\p \rmx_i}}=\frac{\p}{\p \rmx_i}+\frac{1}{4}\Gamma_{ij}^k(\rmx)c^jc^k+h_i(\rmx),
\end{equation}
where  $h_i(\rmx)=\cO(\rmx).$ See \cite[Lemma 4.14]{bgv}.

With this information, we define the Dirac operator on the Clifford bundle $\cE$ by
\begin{equation}\label{diracdef}
\cD^\cE=\sum_{i=1}^nc(e_i)\nabla^\cE_{e_i}
\end{equation}
This produces an elliptic, Fredholm and self-adjoint operator. Let $\cD^\pm=\cD^\cE|_{\cE^\pm}$.  We record here the following crucial identity known as Lichnerowicz formula
\begin{equation}\label{dirac2}
(\cD^\cE)^2=(\nabla^\cE)^*\nabla^\cE+c(F^\cE),
\end{equation}
where $F^\cE$ denotes the curvature of the connection $\nabla^\cE$. The notation $c(F^\cE)$ means that, if $F^\cE=\frac{1}{2}F(e_i,e_j)\,\omega^i\wedge \omega^j$ then $c(F^\cE)=\frac{1}{2}F(e_i,e_j)c(e_i)c(e_j).$ Here we identify $TM$ and $T^*M$ using the metric $g$.

\subsection{The Heat Kernel and Mehler's Formula}
The proofs of the results in this section can be found in \cite[Ch. 3]{ros}  and \cite[Ch. 7]{roe}. The operator $\cD^\cE$ is Fredholm and self-adjoint with discrete spectrum. Using the functional calculus  one can define the smoothing operator $e^{-t(\cD^\cE)^2}$ for every $ t> 0$. This operator has a smooth integral kernel denoted by $k_t(x,y)\in \End (\cE_y, \cE_x)$ for every  $(x,y)\in M\times M$ and $t> 0$, such that for  $\psi\in\Gamma(\cE)$ one has
$$(e^{-t(\cD^\cE)^2}\psi)(x)=\int_Mk_t(x,y)\psi(y) \vol_y,$$
where $\vol_y$ denotes the volume form in the $y$-variable induced from the metric $g.$ The following approximation of the heat kernel on $\mathbb{R}^n$ is going to be used frequently
\begin{equation}\label{heatrn}
q_t(x,y)=\frac{1}{(4\pi t)^{n/2}}\exp\{\frac{-d(x,y)^2}{4t}\},
\end{equation}
where $d(x,y)$ denotes the distance induced by the metric between the points $x$ and $y$ on the manifold. We use $inj$ to denote the injectivity radius of $M$.

The following theorem lists several important properties of the heat kernel $k_t(x,y).$ 
\begin{theo}\label{alpha}
The kernel $k_t(x,y)$ is differentiable in $t$ and smooth in the variables $x,y$.  It also satisfies the following:
\begin{enumerate}
\item If $(\cD^\cE)^2_x$ denotes the operator acting on the $x$-variable, then
$$(\p_t+(\cD^\cE)_x^2)k_t(x,y)=0.$$
\item For every $\psi\in\Gamma(\cE)$ one has
$$\lim_{t\to 0}k_t(x,y)\psi(y)\vol_y\to \psi(x),$$
uniformly in $x.$
\item The operator $e^{-t(\cD^\cE)^2}$ is trace class and its operator trace is given by
\begin{equation}\label{optrace}
\Tr(e^{-t(\cD^\cE)^2})=\int_M \tr_\cE\, k_t(y,y)\vol_y,
\end{equation}
where $\tr_\cE$ denotes the usual trace in $\End(\cE_y).$
\item There are sections $\Theta_j(x,y)\in\End (\cE_y,\cE_x)$ for $j=0,1,2,\ldots$ defined in a neighborhood of the diagonal in $M\times M$ such that for every $N>m=n/2$, there is a constant $C_N$ such that
$$\|\p^k_t(k_t(x,y)-\chi(x,y)q_t(x,y)\sum_{j=0}^Nt^j\Theta_j(x,y))\|_{C^l}\leq C_Nt^{N-m-l/2-k}.$$
Here, $\chi(x,y)=\chi(d(x,y))$ where $\chi(r)=1$ for $r\leq \frac{inj^2}{4} $ and $\chi(r)=0$ for $r\geq inj^2.$ Also, $\|\cdot\|_{C^l}$ denotes the $C^l$ norm on $M\times M$. In other words, there is an asymptotic expansion 
\begin{equation}\label{imprec}
k_t(x,y)\sim \chi(x,y)q_t(x,y)\sum_{j=0}^\infty t^j\Theta_j(x,y)
\end{equation}
\item Using normal coordinates $\{\rmx_l\}_{l=1}^{n}$ centered at $y\in M$ and simplifying notation by using $\Theta_j(\rmx)$ instead of $\Theta_j(x,y)$, one has the following recursion
\begin{equation*}
\nabla_{\p_r}(r^j|g|^{1/4}\Theta_j)=
\begin{cases}
0,&j=0\\
-r^{j-1}|g|^{1/4}(\cD^\cE)^2\Theta_{j-1},&j\geq 1,
\end{cases}
\end{equation*}
where  $|g|=\ddet (g)$. These are ordinary differential equations along geodesics emanating from $y$. The boundary condition $\Theta_0(y,y)=\Theta_0(0)=\mathrm{id_\cE}$ implies that the $\Theta_j$ are uniquely defined for all $j$. Also, the smoothness of the heat kernel implies that $\Theta_0(\rmx)=C|g|^{-1/4}(\rmx).$ For the other values of $j$, one can integrate along the geodesic joining $0$ and $\rmx$ to obtain 
\begin{equation}\label{hkrec}
\Theta_j(\rmx)=-\frac{1}{|g|^{1/4}}\int_0^1s^{j-1}|g|^{1/4}(s\rmx)((\cD^\cE)^2\Theta_{j-1})(s\rmx)\,ds
\end{equation}
\item The famous Mckean-Singer  theorem states that
\begin{equation}\label{singer}
\mathrm{Ind}(\cD^+)=\mathrm{sTr} (e^{-t(\cD^\cE)^2})= \int_M\mathrm{str}\,k_t(y,y)\,\vol_y,
\end{equation}
where the super trace of a map $T\in\End(\cE_y)$ such that $T|_{\cE_y^\pm}\subset \cE_y^\pm$ is defined by $\mathrm{str}(T)=\tr (T|_{\cE_y^+})-\tr (T|_{\cE_y^-})$. In particular, the supertrace of the heat kernel is constant in $t$.
\end{enumerate}
\end{theo}

Finally, we record Mehler's formula which will be used later \cite[Ch. 12]{roe}. Let $R$ be an $n\times n$ antisymmetric matrix with coefficients in a commutative algebra $\cA$ and $F$ be an $N\times N$ matrix with coefficients also in $\cA$. Consider the operator
$$H=-(\p_i+\frac{1}{4}\rmx_j R_{ij})^2+F$$
acting on functions $\bbR^n\to \cA\otimes \End(\bbC^N).$ Then the heat kernel of $H$, denoted by $p_t(\rmx)$ has the following form 
\begin{equation}\label{mehler}
p_t(\rmx)=(4\pi t)^{-n/2}\sqrt{\det\big(\frac{tR/2}{\sinh(tR/2)}\big)}\,e^{-\frac{1}{4t}\langle \rmx|\frac{tR}{2}\coth\frac{tR}{2}|\rmx\rangle}e^{-tF}.
\end{equation}

\subsection{Review of the Proof of the Local Index Theorem}
Let $(M,g)$ be a spin manifold with a fixed spin structure. Let $S=S^+\oplus S^-\to M$ denote the spinor bundle with connection $\nabla^S$ induced by the Levi-Civita connection. Additionally, take $(E, h^E, \nabla^E)$ be a Hermitian bundle with unitary connection. Following the first section, we work with the bundle
$$\cE=S\otimes E=(S^+\otimes E)\oplus (S^-\otimes E).$$
In other words, $\cE^\pm=S^\pm\otimes E.$ The tensor product connection induces a Clifford connection $\nabla^\cE.$ Locally, we have \cite[Sect. 3.2 and 3.3]{bgv}
$$\nabla^\cE_{\frac{\p}{\p \rmx_i}}=\frac{\p}{\p \rmx_i}+\frac{1}{4}\Gamma_{ij}^k(\rmx)c^jc^k+\Gamma^E_i$$

The Dirac operator $\cD^\cE$ is defined as in \eqref{diracdef}, and \eqref{dirac2} becomes
\begin{equation}\label{cool2}
(\cD^\cE)^2=(\nabla^\cE)^*\nabla^\cE+\frac{1}{4}\mathrm{Scal}+c(F^E),
\end{equation}
where $\mathrm{Scal}$ denotes the scalar curvature of $M.$

Before we continue, we review some  standard notation: let $I=(i_1,\ldots, i_k)$ be a multi-index with $1\leq i_1\leq\ldots\leq i_k\leq n$, then $\rmx^I=\rmx_{i_1}\rmx_{i_2}\ldots\rmx_{i_k}$, $\p_\rmx^I=\frac{\p}{\p \rmx_{i_1}}\ldots\frac{\p}{\p \rmx_{i_k}}$ and $c^I=c^{i_1}\ldots c^{i_k}.$ Capital letters as exponents will denote multi-indices from now on. Notice that in $c^I$ we also assume $1\leq i_1<\ldots<i_k\leq n$ since otherwise some of the $c^i$ would simplify using the Clifford algebra rules. The number $|I|=k$ in $c^I$ is the \textit{number of Clifford factors.}

The objective is to compute $\mathrm{sTr} (e^{-t(\cD^\cE)^2})= \int_M\mathrm{str}\,k_t(y,y)\,\vol_y$ which gives the index of $\cD^+$ from Mckean-Singer's theorem.

It is a well-known fact \cite[Lemma 11.5]{roe} that $\mathrm{str} (c^I)=0$ if $|I|\neq n$ and $\mathrm{str}(c^1\ldots c^n)=(-2i)^{n/2}$. For this reason, one needs to identify only those terms in the expansion \eqref{imprec} that have the maximum possible number of Clifford factors and which do not vanish when evaluated at $\rmx=0$, which corresponds to the point $y\in M$ at the center of the normal coordinate system. In other words, for some multi-index $I$, a summand of the form 
$$a\rmx^Ic^1\ldots c^n,$$
with $a\in \End(E_y)$ only contributes to $\mathrm{sTr} (e^{-t(\cD^\cE)^2})$ if $|I|=k=0$ because $\mathrm{sTr}$ also includes an evaluation at $\rmx=0.$

This observation is the motivation of the following idea of Ezra Getzler aimed at extracting the terms mentioned above \cite{get1}.
\begin{define}
The $cG$-grading is defined as follows
\begin{align}
&\cord(\rmx_i)=-1& &\cord(\frac{\p}{\p\rmx_i})=1& &\cord(c^i)=1&
\end{align}
and coefficients in $\End(E_y)$ have $\cord=0.$ For example, $\cord(x^Ic^J\p_\rmx^K)=|J|-|I|+|K|$ and the G-order of a sum is the maximum of the G-orders of the summands. 
\end{define}

\begin{theo}\label{mult1}
Given two differential operators $A,B$ written formally as
\begin{align*}
A&=\sum_{I,J,K} a_{I,J,K}\rmx^Ic^J\p_\rmx^K\\
B&=\sum_{S,T,U}b_{S,T,U}\rmx^Sc^T\p_\rmx^U,
\end{align*}
with $a_{l,I,J}$ and $b_{k,S,T}$ in $\End(E_y)$ independent of $\rmx$ and without Clifford factors.
then $$\cord (A\circ B)\leq \cord(A)+\cord(B).$$
\end{theo}
\begin{proof}
The Clifford factors are constant in a synchronous frame \cite[Lemma 4.13]{bgv}. Consider first the composition of two of the summands, that is 
$$(a_{I,J,K}x^I c^J\p_\rmx^K)\circ(b_{S,T,U}x^S c^T\p_\rmx^U)=a_{I,J,K}b_{S,T,U}x^{I+S} c^{J+T}\p_{K+U}+\ldots$$
where the dots mean summands with fewer derivatives. The sum of multi-indices is defined in the obvious way. In these terms, some derivatives were used differentiating the $x^S$. Each differentiation lowers the $\rmx$-exponent by one or cancels the whole term. For example,  $\p_i(x^P)=0$ if $i\notin P$. Therefore, either  $\cord$-order remains the same (losses one derivative and one factor of $x$) or the term vanishes. Also notice that the number of Clifford factors in $c^{J+T}$ is less than or equal to $|J|+|T|$. This implies the result.

\end{proof}

This definition implies that $\cord (\nabla^\cE_i)\leq 1$, $\cord(\cD^\cE)\leq 2$ and, using formula \eqref{cool2},  $\cord((\cD^\cE)^2)\leq 2.$ 

Since $|g|(y)\neq 0$, we have that $\cord (\Theta_0)\leq 0.$ Using induction, it follows from \eqref{hkrec} that $\cord(\Theta_j)\leq 2j.$

From the Mckean-Singer theorem it follows that $\mathrm{sTr} (e^{-t(\cD^\cE)^2})$ is independent of $t$. Using \eqref{imprec} one gets
\begin{equation}\label{super}
\mathrm{sTr} (e^{-t(\cD^\cE)^2})=(4\pi )^{-n/2}\int_M\mathrm{str}\,\Theta_{n/2}(y,y)\,\vol_y
\end{equation}

A term in $\Theta_{n/2}(y,y)$ with a chance of contributing to the integral on the right hand side of \eqref{super}, after taking super-trace and evaluating at $\rmx=0$ (which is the same as evaluating at $y$), must be of the form $ac^1\ldots c^n$ with $0\neq a\in\End(E_y)$. In this case, $\cord(ac^1\ldots c^n)= n$. Therefore, the  terms with non-vanishing sTr are among those  with $\cord=n$. Theorem \ref{mult1} implies that  the terms of $\cord=n$ in $\theta_{n/2}$ are obtained from taking the terms of maximum $\cord$ in each $\Theta_j$ and following the recurrence \eqref{hkrec} restricted only to those terms.

Basically, this is achieved by starting with the same $\Theta_0$ and following \eqref{hkrec} with $(\cD^\cE)^2$ replaced by its terms of $\cord=2$. That is
\begin{equation}\label{purified1}
H_1=-(\frac{\p}{\p_{\rmx_i}}+\frac{1}{4}\Gamma_{ij}^kc^jc^k)^2+c(F^E)
\end{equation}
This can be simplified even further by noticing that a commutator $[c^ic^j,c^kc^l]$ can have either two Clifford factors or vanish identically. In both cases, the  $\cord$ of the commutator is not the expected $4$ but instead is either $0$ or $2$. Since commutators lower $\cord$ and we only want the terms with the highest possible $\cord$, we can replace the Clifford algebra by the exterior algebra  instead. This means that we need to analyze the heat kernel of the operator
\begin{equation}\label{purified2}
H=-(\frac{\p}{\p_{\rmx_i}}+\frac{1}{4}\rmx_jR_{ij})^2+F^E
\end{equation}

Finally, Mehler's formula and \eqref{purified2} imply that 
\begin{equation}
\mathrm{ind}(\cD^+)=\mathrm{sTr} (e^{-t(\cD^\cE)^2})=(4\pi )^{-n/2}(-2i)^{n/2}\int_M\sqrt{\det\big(\frac{R/2}{e^{R/2}-e^{-R/2}}\big)}\,e^{-F}\,\vol,
\end{equation}
from which the index theorem follows. See \cite[pp. 162]{roe}.

\section{Heat Kernels and Tensor Powers of Line Bundles}

In this section we give a new and simple proof of Theorem 1.6.1 in \cite{mm}.   Our argument  only uses the expansion \eqref{imprec} and a small modification of the  $cG$-grading  explained in the previous section.
\subsection{Statement of the Theorem}\label{exp1}
Let $(M,g)$ be a complex manifold whose almost complex structure is compatible with the Riemannian metric. Let $(L,h^L)$ be a holomorphic line bundle over $M$ with Chern connection $\nabla^L$, and $(E,h^E)$ be holomorphic vector bundle over $M$ with Chern connection $\nabla^E.$ Let $T^{*(0,1)}M$ be the $(0,1)$-part of the cotangent bundle of $M$ with the induced connection coming from the Levi-Civita connection. Similarly, one can induce a connection on the wedge powers $\Lambda^{*}T^{*(0,1)}M.$ Given $y\in M$, we use these connections to generate synchronous local orthonormal frames on each of these bundles.

In this setup, the bundle to be considered is
\begin{equation}\label{mmbundle}
\cE_p=\Lambda T^{*(0,1)}M\otimes L^{\otimes p}\otimes E,
\end{equation}
where $L^{\otimes p}$ denotes the $p^{th}$ tensor power of $L$ with itself. The bundle $\cE_p$ is a Clifford bundle with Clifford multiplication given by
$$c(e_j)=\sqrt{2}\epsilon(\omega^{(0,1)j})-\sqrt{2}\iota_{e_j^{(0,1)}},$$
where $\epsilon$ and $\iota$ denote exterior and interior product respectively and $\omega^{(0,1)j}$, $e_j^{(0,1)}$ denote the $(0,1)$ parts of the covector and the vector respectively. Also, we use the tensor product of the connections on $\Lambda^{*}T^{*(0,1)}M$, $L$ and $E$ to generate a Clifford connection $\nabla^{\cE_p}$ on $\cE_p.$

The associated geometric Dirac operator to $\nabla^{\cE_p}$ and the Clifford product above is defined as in \eqref{diracdef} and will be denoted by $\cD_p$ to simplify notation.

Let $F^L$ denote the curvature of $\nabla^L$, then the endomorphism $\dot{F}^L$ is defined as follows: take $X,Y$ vector fields in $T^{(1,0)}M$ then
$$F^L(X,\bar{Y})=g(\dot{F}^LX,\bar{Y}).$$
Notice that $\dot{F}^L\in\End({T^{(1,0)}M}).$ Set $\tau(y)=\tr (\dot{F}^L|_{y}) $ for any $y\in M.$ Also, if $\{Z_j\}_{j=1}^m$ is any orthonormal frame of $T^{(1,0)}M$ then one defines $$\omega_d=-F^L(Z_l,\bar{Z}_m)\epsilon(\bar{Z}^l)\iota_{\bar{Z}_m}, $$
where $\bar{Z}^l\in T^{(0,1)*}M$ is the dual covector. With all this notation, we can state the main theorem
\begin{theo}\label{statement1}
\cite[thm 1.6.1]{mm} For each $u>0$ and all $k\in\mathbb{N}$ we have as $p\to\infty$ that
\begin{equation}\label{imp}
e^{-\frac{u}{p}\cD_p^2}(y,y)=k_{u/p}(y,y)=(2\pi)^{-n/2}p^{n/2}\frac{e^{2u\omega_d}\det(\dot{F}^L)}{\det{(1-e^{-2u\dot{F}^L}})}\otimes \mathrm{id}_E+o(p^{n/2})
\end{equation}
\end{theo}
This asymptotic expansion is a key step in the proof of the holomorphic Morse inequalities explained in \cite[section 1.7]{mm}. The apparent discrepancy in the exponents of $p$ occurs because we denote by $n$ the real dimension of $M$ while the authors in \cite{mm} use $n$ to denote the complex dimension.

\subsection{The Modified Grading and the Proof}
Now we explain the modification of the grading scheme that gives another proof of theorem \ref{statement1}. First of all, notice that \eqref{imp} does not involve traces or super traces. For this reason, counting the number of Clifford factors of the form $c^j$ does not give any relevant information. Instead, what we really care about is counting the exponents of $p$.


\begin{define}
The $p$-grading is defined as follows
\begin{align}
&\oord(\rmx_j)=-1& &\oord(\frac{\p}{\p\rmx_l})=1& &\oord(p)=2&
\end{align}
and coefficients in $\End(\Lambda T_y^{*(0,1)}M\otimes L_y\otimes E_y)$ have $\oord=0.$ The $p$-grading of a sum is the maximum of the $p$-gradings of the summands. 
\end{define}
The proof of the following result is almost identical to that of theorem \ref{mult1}.

\begin{theo}\label{mult2}
Given two differential operators $A,B$ written formally as
\begin{align*}
A&=\sum_{l,I,J} a_{l,I,J}p^l\rmx^I\p_\rmx^J\\
B&=\sum_{k,S,T}b_{k,S,T}p^k\rmx^S\p_\rmx^T,
\end{align*}
with $a_{l,I,J}$ and $b_{k,S,T}$ in $\End(\Lambda T_y^{*(0,1)}M\otimes L_y\otimes E_y)$ independent of $p.$
then $$\oord (A\circ B)\leq \oord(A)+\oord(B).$$
\end{theo}

Locally, the connection $\nabla^{\cE_p}$ can be written as
\begin{equation}\label{almost}
\nabla^{\cE_p}_{\frac{\p}{\p\rmx_i}}=\frac{\p}{\p\rmx_i}+\frac{1}{4}\Gamma_{ij}^kc^jc^k+\frac{1}{2}\Gamma^{\det}_i+p\Gamma^L_i+\Gamma^E_i,
\end{equation}
where $\Gamma^{\det}$ is the connection one-form on the determinant line bundle $\det(T^{(1,0)}M).$
Notice that, with respect to the decomposition \eqref{nablaloc}, we have $h_i(\rmx)=\frac{1}{2}\Gamma^{\det}_i+p\Gamma^L_i+\Gamma^E_i $ and $\oord(\frac{1}{2}\Gamma^{\det}_i+\Gamma^E_i )\leq -1.$ This implies that 
$$\oord (\nabla^{\cE_p}_{\frac{\p}{\p\rmx_i}})\leq 1,$$
and its highest order summands are $\frac{\p}{\p\rmx_i}+p\Gamma^L_i.$ Theorem \ref{mult2} implies then that $\oord(\cD^2_p)\leq 2$ and the recurrence \eqref{hkrec} implies that $\oord(\Theta_j)\leq 2j.$ The summands in $\Theta_j$ of $\oord=2j$ that do not vanish on the diagonal will have $j$ as the highest possible exponent of $p.$

From the expansion \eqref{imprec} we get
\begin{equation}\label{cool}
k_{u/p}(y,y)\sim \frac{p^{n/2}}{(4\pi u)^{n/2}}\sum_{j=0}^\infty \frac{u^j}{p^j}\Theta_j(y,y),
\end{equation}
which implies that the highest possible exponent of $p$ in \eqref{cool} is $n/2.$ This exponent can be reached only using the terms of highest $\oord$ in each $\Theta_j$.

In other words, we need to investigate the heat kernel of the operator $$-(\frac{\p}{\p\rmx_i}+p\Gamma^L_i)^2+pc(F^L)=-(\frac{\p}{\p\rmx_i}+p\Gamma^L_i)^2-p(2\omega_d+\tau)$$
(See \cite[pp. 46]{mm}). Using \eqref{sync} and dropping terms of lower $\oord$ clarifies that we need to understand the heat kernel of \footnote{We are omiting  $\otimes \mathrm{id}_E$ in this notation.}
\begin{equation}\label{clean}
Q_p=-(\frac{\p}{\p\rmx_i}-\frac{p}{2}\rmx_jF^L_{ij})^2-p(2\omega_d+\tau)
\end{equation}

Using Mehler's formula again, the integral kernel  of $e^{-tQ_p^2}$  is given by
$$e^{-tQ_p^2}(y,y)=\mathrm{id_E}\otimes (2\pi t)^{-n/2}\sqrt{\det\big(\frac{tpF^L}{e^{tpF^L}-e^{-tpF^L}}\big)}\,e^{tp(2\omega_d(y)+\tau(y))}$$
We have that $\det(F^L)=\det^2(\dot{F}^L)$ so we get
$$e^{-tQ_p^2}(y,y)=\mathrm{id_E}\otimes (2\pi t)^{-n/2}{\det\big(\frac{tp\dot{F}^L}{1-e^{-2tp\dot{F}^L}}\big)}\det(e^{-tp\dot{F}^L})\,e^{tp(2\omega_d(y)+\tau(y))}$$
Also, $\det(e^{tp\dot{F}^L})|_y=e^{\tau(y)}$ and the corresponding terms cancel. Therefore,
$$
e^{-tQ_p^2}(y,y)=\mathrm{id_E}\otimes (2\pi t)^{-n/2}{\det\big(\frac{tp\dot{F}^L}{1-e^{-2tp\dot{F}^L}}\big)}\,e^{2tp\omega_d(y)}
$$
Finally, replace $t$ by $u/p$ and notice that $\det{u\dot{F}^L}=u^{n/2}\det{\dot{F}^L}$. This gives the leading order term in \eqref{imp} and proves theorem \ref{statement1}.

\section{Line Bundles and Heat Kernels on Odd-dimensional Manifolds}
In this section we show one more application of the grading technique by giving an alternative proof of \cite[Formula (4.1)]{sav}. One major difference in this case is that the base manifold $Y$ is odd-dimensional. Again, a small modification of the definition of $\cord$ will give the leading order term of the asymptotic expansion of a heat kernel in powers of a certain parameter.

\subsection{Definitions and Statement of the Theorem}\label{exp2}
Let $(Y,g)$ be a Riemannian spin manifold with $\dim_\bbR=n=2m+1.$ One fixes the spin structure on $Y$ and denotes by $S\to Y$ the associated spinor bundle with induced connection $\nabla^S$. Let $L\to Y$ be a Hermitian line bundle with a unitary connection $\nabla_0=d+A_0$. Let $a$ be a $i\bbR$-valued one-form on $Y$. Define a family of connections on $S\otimes L$ by 
$$\nabla^r=\nabla^S\otimes \id_L+\id_S\otimes (\nabla^0+ra),$$ 
where $r$ is a real parameter. 
Using the family $\nabla^r$, we define a family of Dirac operators by
\begin{equation}\label{diracodd}
\cD^r=c^k\,\nabla^r_{e_k}
\end{equation} 
as in \eqref{diracdef}. Notice that the bundle $\cE=S\otimes L$ in this case is not $\mathbb{Z}_2$-graded.

The heat kernel $e^{-t(\cD^r)^2}$ is defined using functional calculus as before and all the statements in theorem \ref{alpha} remain valid with the exception of the McKean-Singer formula. Also, Mehler's formula remains valid with $n$ odd.
In the next section, we present a new proof of the following result
\begin{theo}\label{last}
\cite[theorem 4.1]{sav} For any fixed $t>0$, one has
\begin{equation}\label{limit}
\lim_{r\to\infty}r^{-n/2}\Tr(e^{-\frac{t}{r}(\cD^r)^2})=(4\pi t)^{-n/2}\int_Y\sqrt{\det\dfrac{tA_y}{\tanh tA_y}}\,\vol_y
\end{equation}
\end{theo}
Here $A_y\in \End(T_yY)$ is defined by $ida(X_1,X_2)(y)=g(X_1, AX_2)$ for all $X_1,X_2\in T_yY.$ The convergence is uniform on compact intervals of $t\in (0,\infty).$

\subsection{Modification of the Grading and Proof}
Fix $y\in Y$ and use this point as the center of a normal coordinate system. Around $y$, define  a synchronous orthonormal frame on $L$ by using $\nabla^0$ and parallel transport along geodesics emanating from $y$.  We use the same notation $\{\rmx_l\}_{l=1}^n$ as before to denote normal coordinates.

Locally, the connection $\nabla^r$ has the form
$$\nabla^r_{\frac{\p}{\p_{\rmx_k}}}=\frac{\p}{\p_{\rmx_k}}+\Gamma_k^S+A^0_k+ra_k,$$
with $\Gamma^S_k, A^0_k\in \cO(\rmx)$.
Also,
\begin{equation}\label{boch2}
(\cD^r)^2=\nabla^{r*}\nabla^r+c(F_0)+rc(da)+\frac{1}{4}\mathrm{Scal}
\end{equation}
Now we adapt the grading to count powers of $r$.

\begin{define}
The $r$-grading is defined as follows
\begin{align}
&\rord(\rmx_j)=-1& &\rord(\frac{\p}{\p\rmx_j})=1& &\rord(r)=1&
\end{align}
and coefficients in $\End(S_y\otimes L_y)$ have $\rord=0.$ The $r$-grading of a sum is the maximum of the $r$-gradings of the summands. 
\end{define}
Notice for example that $\rord(ra_k)\leq1$ since the 1-form $a=a_jd\rmx^j$ might not vanish at $\rmx=0.$
The proof of the following result is almost identical to that of theorem \ref{mult1}.

\begin{theo}\label{mult3}
Given two differential operators $A,B$ written formally as
\begin{align*}
A&=\sum_{l,I,J} a_{l,I,J}r^l\rmx^I\p_\rmx^J\\
B&=\sum_{k,S,T}b_{k,S,T}r^k\rmx^S\p_\rmx^T,
\end{align*}
with $a_{l,I,J}$ and $b_{k,S,T}$ in $\End(S_y\otimes L_y)$ independent of $r.$
then $$\rord (A\circ B)\leq \rord(A)+\rord(B).$$
\end{theo}

We have that $\rord(\nabla^r_{\frac{\p}{\p_{\rmx_k}}})\leq1$, $\rord (\cD^r)\leq 1$ and $\rord((\cD^r)^2)\leq 2$. Similarly, it follows from \eqref{hkrec} that $\rord (\Theta_0)=0$ and $\rord(\Theta_j)\leq 2j.$ The summands in $\Theta_j$ of $\rord=2j$ that do not vanish on the diagonal will have $j$ as the highest possible exponent of $r.$

From the expansion \eqref{imprec} we get
\begin{equation}\label{cool10}
e^{-\frac{t}{r}(\cD^r)^2}(y,y)=k_{t/r}(y,y)\sim \frac{r^{n/2}}{(4\pi t)^{n/2}}\sum_{j=0}^\infty \frac{t^j}{r^j}\Theta_j(y,y),
\end{equation}
which implies that the highest possible exponent of $r$ in \eqref{cool} is $n/2.$ This exponent can be reached only using the terms of highest $\rord$ in each $\Theta_j$. Therefore, we need to find the formula for the heat kernel of 
\begin{equation}\label{distilled}
H_r=-(\frac{\p}{\p\rmx_k}+ra_k)^2+rc(da)
\end{equation}
Mehler's formula gives
\begin{equation}\label{last1}
\frac{1}{r^{n/2}}e^{-\frac{t}{r}H_r}=\frac{1}{(4\pi t)^{n/2}}\sqrt{\det\dfrac{tA_y}{\sinh tA_y}}e^{-tc(da)}+o(r^{-1/2})
\end{equation}
Integrating \eqref{last1} over $Y$, it follows (see \cite[page 859]{sav}) that this expression is equivalent to the one written in \eqref{limit}.

\end{document}